\documentclass[letterpaper, 10 pt, conference]{ieeeconf}  

\IEEEoverridecommandlockouts                              

\usepackage{amsmath,mathrsfs} 
\usepackage{amssymb}  
\usepackage{graphicx}
\usepackage{cite}
\usepackage{url}
\usepackage{calligra}
\usepackage{subfigure}
\usepackage{float}
\newtheorem{lemma}{Lemma}
\newtheorem{theorem}{Theorem}
\newtheorem{proposition}{Proposition}
\newtheorem{assumption}{Assumption}

\newtheorem{remark}{Remark}
\newenvironment{IEEEproof}{{\bf Proof.}}{ }
 \usepackage{xcolor}
\renewcommand{\baselinestretch}{0.8367}
\title{\LARGE \bf
Towards Co-Designed Event-Triggered Extremum Seeking
}

\author{Roberto Luo$^{1}$, Pedro Henrique Silva Coutinho$^{1}$, Victor Hugo Pereira Rodrigues$^{1}$,\\ Tiago Roux Oliveira$^{1}$,  
and Miroslav Krstic$^{2}$
\thanks{$^{1}$ R. Luo, P. H. S. Coutinho, V. H. P. Rodrigues, and T. R. Oliveira are with the Department of Electronics and Telecommunication Engineering (DETEL), State University of Rio de Janeiro (UERJ), Brazil. 
        Email: {\tt robertoluoo@gmail.com; [phcoutinho, victor.rodrigues]@eng.uerj.br, tiagoroux@uerj.br}}
        \thanks{$^{2}$M. Krstic is with the Department of Mechanical and Aerospace
Engineering (MAE), University of California at San Diego (UCSD),  USA.
        Email: {\tt  mkrstic@ucsd.edu}}
}

\begin{document}

\maketitle
\thispagestyle{empty}
\pagestyle{empty}

\begin{abstract}
This paper studies event-triggered gradient-based multivariable extremum seeking for nonlinear maps with polytopic Hessian uncertainty. Unlike existing event-triggered extremum-seeking methods, which first fix the controller (typically diagonal) and then design the triggering mechanism, the proposed approach jointly synthesizes the controller and the triggering mechanism through a co-design framework that admits both diagonal and full controller gain matrices. The co-design problem is formulated as a convex optimization problem with linear matrix inequality constraints. Its solution guarantees exponential convergence of the average closed-loop system with a prescribed decay rate while maximizing the admissible triggering threshold to reduce communication. Lyapunov and averaging analyses establish exponential stability of the event-triggered system, and Zeno-freeness is proved to guarantee implementability. Numerical results illustrate that diagonal gains cannot achieve the same triggering thresholds and decay rates as the full controller gain matrices, highlighting the benefits of exploiting Hessian coupling information. 
\end{abstract}

\section{INTRODUCTION}
Extremum seeking (ES) is a model-free optimization technique that drives a system output toward its extremum. Through periodic perturbation signals, ES estimates the gradient of an unknown objective function, enabling real-time optimization without requiring an explicit model of the function \cite{KW:2000}.

Since its introduction, ES control has been extensively studied from both theoretical and practical perspectives \cite{AS:2024}. Applications include online tuning of PID controllers in neuromuscular electrical stimulation systems 
\cite{POPF:2020} and traffic congestion control \cite{SHPMG:2023}. 
On the theoretical side, recent advances include ES with uncertainty estimation \cite{GUAY:2021}, continuous-time and sampled-data ES with delays \cite{OTK:2017,ZFO:2023}, ES for PDE systems \cite{TRoux:2022}, and 
the extension to Nash equilibrium seeking in noncooperative games with heterogeneous agent dynamics \cite{CSm:2026}.

In parallel, event-triggered control (ETC) has emerged as an effective strategy to reduce resource usage in feedback systems. Unlike traditional periodic control, ETC updates the control input only when a triggering condition is satisfied, reducing computation, communication, and energy consumption \cite{ZHGDDYP:2020}. Early studies in \cite{AB:1999} compared periodic and event-triggered implementations, showing improved computational efficiency with event-based strategies. This framework was further formalized in \cite{T:2007}.

In this context, the combination of ES and event-triggered control has shown promising results for the optimization of scalar and multivariable static maps \cite{VHPR:2025,ROKT:2026}. However, most existing schemes adopt an emulation-based design, in which the controller gain is designed first, and the event-triggering mechanism (ETM) is synthesized afterward. Although this approach simplifies the design process, it restricts the achievable closed-loop performance because the controller and the triggering law cannot be optimized simultaneously. Moreover, prescribing a diagonal controller gain restricts the available design space because it cannot exploit the coupling between optimization variables induced by the Hessian. Likewise, most existing event-triggered extremum-seeking schemes adopt an emulation-based design, in which the controller gain is fixed a priori and only the triggering mechanism is optimized. Consequently, the controller and the triggering law cannot be jointly optimized, often leading to conservative triggering conditions and unnecessary communication. 

\textcolor{black}{This paper addresses these limitations through a unified \emph{co-design} framework that jointly synthesizes the controller gain and the ETM \cite{APDNH:2018,CPBPP:2026}. By exploiting prior information about the objective-function curvature, the proposed framework accommodates both diagonal and full-matrix controller gains, allowing the co-design to exploit Hessian-induced cross-couplings that are unavailable to diagonal designs. Unlike diagonal gains, which treat each optimization coordinate independently, full-matrix gains allow larger admissible transmission errors and consequently fewer triggering events. As illustrated in Section~\ref{sec:sim_siso}, the resulting full-matrix co-design achieves lower objective values and fewer triggering events than both the diagonal co-design and conventional emulation-based designs.} 
\textcolor{black}{From this perspective, prior information about the objective-function curvature is not only valuable for guaranteeing stability, but also for reducing communication by enlarging the feasible joint design space of the controller gain and triggering mechanism.}

\emph{Notation:} Throughout the paper, $\|\cdot\|$, $|\cdot|$, and $\operatorname{tr}(\cdot)$ denote the Euclidean norm and its induced matrix norm, the absolute value, and the trace operator, respectively. For a positive definite matrix, $\lambda_{\min}(\cdot)$ and $\lambda_{\max}(\cdot)$ denote its minimum and maximum eigenvalues, respectively. Moreover, $\mathcal{O}(\cdot)$ denotes the standard big-$O$ notation; namely, $\delta_{1}(\varepsilon)=\mathcal{O}(\delta_{2}(\varepsilon))$ if there exist positive constants $k$ and $c$ such that $|\delta_{1}(\varepsilon)| \le k |\delta_{2}(\varepsilon)|$, for all $|\varepsilon|<c$. For a symmetric matrix $X$, $X \succ 0 (\prec 0)$ denotes that $X$ is a positive (negative) definite matrix.

\section{Problem Formulation} \label{sec:prblFrm_siso}

Fig.~\ref{fig:BD_SET_GradientES_SISO} shows the setup of the event-triggered ES control system addressed in this work, whose convergence properties are strongly influenced by the curvature of the objective function $Q(\cdot)$.
\begin{figure}[ht!]
\centering
\includegraphics[width=0.75\columnwidth]{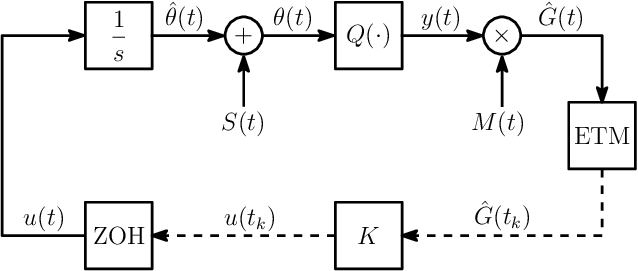}
\caption{Block diagram of the event-triggered gradient-based extremum seeking control.}
\label{fig:BD_SET_GradientES_SISO}
\end{figure}
\subsection{Static Map, Probing and Demodulation Signals}
A multivariable quadratic static map is considered
\begin{align}
\!\!\!y(t)&=Q(\theta(t)) = Q^{\ast}+\frac{1}{2}(\theta(t)-\theta^{\ast})^{\top}H^{\ast}(\theta(t)-\theta^{\ast})\,, \label{eq:y_v2}
\end{align}
where $Q^{\ast} \in \mathbb{R}$ is the unknown optimal value, $H^{\ast} = H^{\ast \top} \in \mathbb{R}^{n \times n}$ is the Hessian matrix, $\theta^{\ast} \in \mathbb{R}^{n}$ is the optimizer, and $\theta(t) \in \mathbb{R}^{n}$ is the input vector. Since $H^{\ast}$ is unknown, it may be either positive or negative definite depending on whether the problem is a minimization or maximization task. This sign is reflected in the choice of the controller gain $K$: \textcolor{black}{for maximization, $K\succ0$ when $H^{\ast}\prec0$, while for minimization, $K\prec0$ when $H^{\ast}\succ0$.} The probing-perturbation signal $S(t)$, 
and the demodulation signal $M(t)$ are defined as in \cite{GKN:2012}: 
\begin{align}
S(t)&= \left[a_{1}\sin\left(\omega_1 t\right),\ldots,a_{i}\sin\left(\omega_i t\right)\right]^{\top}\!\!, \label{eq:S_v1} \\
M(t)&=2\left[\frac{1}{a_{1}}\sin\left(\omega_1 t\right),\ldots,\frac{1}{a_{i}}\sin\left(\omega_i t\right)\right]^{\top}\!\!, \label{eq:M_v1}
\end{align}
where $a_{i}$ denote the perturbation amplitudes and the perturbation frequencies $\omega_{i}$'s are chosen as,
\begin{align}
\omega_{i}=\omega_{i}'\omega \,, \quad i \in \left\{1,\ldots\,,n\right\}\,, \label{eq:omegai_event}
\end{align}
satisfying the following assumption.
\begin{assumption}\label{hipotese_1}
The perturbation frequencies $\omega_{i}$ satisfy
\begin{align}
\omega'_{i} 	\notin \left\{\omega'_{j}\,,~\frac{1}{2}(\omega'_{j}+\omega'_{k})\,,~\omega'_{j}+2\omega'_{k}\,,~\omega'_{k}\pm\omega'_{l}\right\}\,, \label{eq:omega_iNotIn}
\end{align}
for all $i$, $j$, $k$ and $l$, where $\omega$ is a positive constant and $\omega_{i}'$ is a rational number.
\end{assumption}




The parameters of the nonlinear static map \eqref{eq:y_v2} are unknown, while the output $y(t)$ is measurable. The input $\theta(t)$ is constructed from the estimate $\hat{\theta}(t)\in\mathbb{R}^n$ of $\theta^\ast$ and a periodic perturbation $S(t)$, as follows
\begin{align}
\theta(t)&=\hat{\theta}(t)+S(t)\,. \label{eq:theta_event_siso}
\end{align}
In general, stability analyses of ES control systems assume that the sign of the Hessian matrix $H$ is known, allowing the gain matrix $K$ to be selected with the opposite sign. In this work, a more general uncertainty description is adopted by assuming that the unknown Hessian belongs to a polytopic domain:
\begin{assumption}\label{hipotese_2}
The unknown Hessian matrix $H$ belongs to the following polytopic domain:
\begin{align}
\mathcal{H} \in \mathrm{co}\{H_{1},\ldots,H_{N}\},
\end{align}
where $N$ is the number of polytope vertices and $H_{i}$, $i = 1,\ldots,N$, are known matrices.
\end{assumption}
Regarding Assumption~\ref{hipotese_2}, any Hessian matrix $H^{\ast} \in \mathcal{H}$ admits the parameterization
\begin{align}
H^* = H(\alpha) = \sum_{i = 1}^{N} \alpha_{i}H_{i}\,, \label{eq:conv_h}
\end{align}
where $\alpha = (\alpha_1, \ldots, \alpha_N)$ is a vector of fixed but unknown parameters belonging to the unit simplex.
\begin{align}
\Xi = \left\{ \alpha \in \mathbb{R}^N {:} \sum_{i=1}^{N} \alpha_i = 1,\; \alpha_i \ge 0,\; i = 1, \ldots, N \right\}.
\end{align}
The Hessian matrix polytope can be obtained using different uncertainty representations. For example, if it is assumed that $H^\ast = (1 + \sigma)H_0$,
where $H_0$ is a known nominal Hessian estimate and $\sigma$ is an uncertain parameter such that $|\sigma| \le \bar{\sigma}$, this parameterization can be incorporated into a two-vertex polytope with $H_1 = (1 - \bar{\sigma})H_0$ and $H_2 = (1 + \bar{\sigma})H_0$.

\subsection{Estimation Error and Gradient Estimate}

Consider the vector $\hat{\theta}(t) \in \mathbb{R}^{n}$ as an estimate of $\theta^{\ast}$, and define the \textit{estimation error} as
\begin{align}
\tilde{\theta}(t):=\hat{\theta}(t)-\theta^{\ast}\,. \label{eq:thetaTilde}
\end{align}
Moreover, the \textit{gradient estimate} is given by
\begin{align}
\hat{G}(t) = M(t)y(t) \in \mathbb{R}^{n} \label{eq:grad_est}\,,
\end{align}
\noindent using \eqref{eq:thetaTilde} to eliminate $\hat{\theta}(t)$ from \eqref{eq:theta_event_siso}, and substituting the resulting expression into \eqref{eq:y_v2}, yields
\begin{align}
y(t) = Q^{\ast}+\frac{1}{2}(\tilde{\theta}(t)+S(t))^{\top}H^{\ast}(\tilde{\theta}(t)+S(t))\,.
\label{eq:y_v4}
\end{align}
Using \eqref{eq:y_v4} in \eqref{eq:grad_est}, the gradient estimate is obtained as
\begin{align}
\hat{G}(t)& {=}M(t)\frac{1}{2}\tilde{\theta}^{\top}(t)H^{\ast}\tilde{\theta}(t) {+}\left(H^{\ast} {+}\Delta \! \mathscr{H}^{\ast}\!\!(t)\right)\tilde{\theta}(t) {+}\Delta(t)\,,\label{eq:hatG_20240302_2}
\end{align}
where  $\frac{d\Delta \! \mathscr{H}^{\ast}}{dt}(t)$ and $\Delta(t)$ consist of sinusoidal terms. Since the term $\tilde{\theta}^{\top}(t)H^{\ast}\tilde{\theta}(t)$ is quadratic in $\tilde{\theta}(t)$ and can therefore be neglected in a local analysis around the optimizer \cite{AK:2003}, the gradient estimate (\ref{eq:hatG_20240302_2}) can be rewritten as
\begin{align}
\hat{G}(t)&=\left(H^{\ast}+\Delta \! \mathscr{H}^{\ast}(t)\right)\tilde{\theta}(t)+\Delta(t)\,. \label{eq:hatG_20240302_3}
\end{align}

On the other hand, from the time derivative of (\ref{eq:thetaTilde}) and the block diagram shown in Fig.~\ref{fig:BD_SET_GradientES_SISO}, the dynamics governing $\hat{\theta}(t)$ and consequently $\tilde{\theta}(t)$, are given by
\begin{align}
\frac{d\tilde{\theta}}{dt}(t)&=\frac{d\hat{\theta}}{dt}(t)=u(t) \label{eq:u_continuous}\,,
\end{align}
with $u(t)= [u_{1}(t),u_{2}(t),\ldots,u_{n}(t)]^{\top} \in \mathbb{R}^{n}$. Furthermore, using (\ref{eq:u_continuous}), the time derivative of (\ref{eq:hatG_20240302_3}) is given by
\begin{align}
\frac{d\hat{G}}{dt}(t)& {=}\left(H^{\ast} {+}\Delta \mathscr{H}^{\ast}(t)\right)u(t) {+}\frac{d\Delta  \mathscr{H}^{\ast}}{dt}(t)\tilde{\theta}(t) {+}\frac{d \Delta}{dt}(t)\,. \label{eq:dhatgdt}
\end{align}
The derivative terms $\frac{d\Delta \! \mathscr{H}^{\ast}}{dt}$ and $\frac{d\Delta}{dt}$ in \eqref{eq:dhatgdt} are bounded periodic signals with zero mean over one excitation period. Their contribution therefore vanishes in the average model obtained under the standard frozen-state averaging argument.

\subsection{Event-Triggered Extremum Seeking Control}
Let $t_{k}$ denote an unbounded and monotonically increasing sequence of time instants
\begin{align}
0{=}t_{0}<t_{1}<\ldots<t_{k}<\ldots\,, \quad k \in \mathbb{N}\,, \lim_{k \to \infty} t_{k}{=}\infty \,, \label{eq:s_k_event_siso}
\end{align}
with aperiodic sampling intervals $\tau_{k}=t_{k+1}-t_{k}>0$.
A continuous measurement of the system output is assumed, while an event-triggered approach is employed. The actuator converts the discrete-time control input $u_k=U(t_{k})$ into a piecewise-continuous input $u(t)$ with a zero-order hold. Then, the control law is given by
\begin{align}
u(t)\!=\!K\hat{G}(t_{k})\,,
\quad 
\forall t \!\in\! \lbrack t_{k}\,, t_{k+1}\phantom{(}\!\!) \,, \quad k\!\in\! \mathbb{N}  \,. \label{eq:U_event_siso}
\end{align}
As control updates occur at discrete time instants rather than continuously, the following transmission error is induced:
\begin{align}
e(t):=\hat{G}(t_{k})-\hat{G}(t) \,, \quad \forall t \in \lbrack t_{k}\,, t_{k+1}\phantom{(}\!\!) \,, \quad k\in \mathbb{N} \,. \label{eq:e_event_siso}
\end{align}
Therefore, the control law (\ref{eq:U_event_siso}) can be rewritten in terms of the error (\ref{eq:e_event_siso}) as
\begin{align}
u(t)=K\hat{G}(t)+Ke(t)\,, ~ \forall t \in \lbrack t_{k}\,, t_{k+1}\phantom{(}\!\!) \,, ~ k\in \mathbb{N}  \,. \label{eq:U_event_v2_siso}
\end{align}
By substituting \eqref{eq:U_event_v2_siso} into \eqref{eq:u_continuous} and \eqref{eq:dhatgdt}, the closed-loop dynamics of $\hat{G}(t)$ and $\tilde{\theta}(t)$ are obtained as functions of the transmission error vector $e(t)$
\begin{align}
&\frac{d\hat{G}}{dt}(t) {=}\left(H^{\ast} {+}\Delta \! \mathscr{H}^{\ast}\!(t)\right)K\hat{G}(t) +\frac{d\Delta \! \mathscr{H}^{\ast}}{dt}\!(t)\tilde{\theta}(t) \nonumber\\
& \quad ~~~~~~~~~ + \left(H^{\ast} {+}\Delta \! \mathscr{H}^{\ast}\!(t)\right)Ke(t) {+}\frac{d \Delta}{dt}(t)\,, \label{eq:dotHatGav_event_3_siso} \\
&\frac{d\tilde{\theta}}{dt}(t) {=} K\left(H^{\ast}{+}\Delta \! \mathscr{H}^{\ast}(t)\right)\tilde{\theta}(t){+}K\Delta(t){+}Ke(t). 
\label{eq:dotTildeTheta_2_event_siso}
\end{align}
In this work, the transmission sequence is generated by the following ETM:
\begin{align}
    \hspace{-0.5cm}t_{k+1}&{=}\inf\left\{t\in\mathbb{R}^{+}: t>t_{k} \wedge  \Gamma (\hat{G}(t), e(t)) < 0 \right\},
    \label{eq:tk+1_event} 
\end{align}
where the trigger function is
\begin{align}
\Gamma(\hat{G}(t),e(t)) = \hat{G}^{\top}(t)Q_{\rm{G}}\hat{G}(t) -e^{\top}(t)Q_{\rm{e}}e(t).   \label{eq:gamma}
\end{align}

\subsection{Average Closed-Loop~System}
For the stability analysis, introduce the scaled time
$\bar t=\omega t$, where $\omega=2\pi/T$ and $T$
denotes the common period of the perturbation signals \cite{GKN:2012}: 
\begin{align}
T = 2\pi \times \mathrm{LCM}\left(\tfrac{1}{\omega_{1}}, \ldots, \tfrac{1}{\omega_{n}} \right)\,.
\label{eq:omega_event_1_siso}
\end{align}

Defining the augmented state $y=[\hat G^\top,\tilde\theta^\top,e^\top]^\top$, the closed-loop dynamics in \eqref{eq:dotHatGav_event_3_siso} and \eqref{eq:dotTildeTheta_2_event_siso} can be written in the following form
\begin{align}
\frac{dy}{d\bar t}(\bar{t})
&\in
\frac1\omega F\!\left(\bar t,y\right), \quad \bar{t} \neq \bar{t}_k(y),\label{eq:dotX_event_siso} \\
\Delta y|_{\bar{t} = \bar{t}_k(y)} & \in \frac1\omega \mathcal{I}_k(y),
\end{align}
where $\mathcal{F} = [\mathcal{F}_1 \; \mathcal{F}_2 \; \mathcal{F}_3]^\top$, $\mathcal{I}_k = [0 \; 0 \; - e^\top]^\top$, $\mathcal{F}_{1}$ and $\mathcal{F}_{2}$ are obtained directly from \eqref{eq:dotHatGav_event_3_siso} and \eqref{eq:dotTildeTheta_2_event_siso}, respectively, $\mathcal{F}_3 = -\mathcal{F}_1$. All these functions are $T$-periodic in $\bar t$. Due to the discontinuous nature of the proposed control strategy, the averaging method for discontinuous systems, as developed in \cite{P:1979}, is employed.
The augmented-state system $y$ is characterized by the presence of the small parameter $1/\omega$ and a $T$-periodic function $\mathcal{F}\left(\bar{t},y\right)$ with respect to $\bar{t}$. Therefore, the averaging method is employed to analyze the stability of $\mathcal{F}\left(\bar{t},y,\frac{1}{\omega}\right)$ in the limit $\displaystyle \lim_{\omega\to\infty}\frac{1}{\omega}=0$, as described in \cite{P:1979,KPS:2012,VHPR:2025,ROKT:2026}. Consequently, the associated average system~is
\begin{align}
\frac{dy_{\rm av}}{d\bar t}
=
\frac1\omega
\mathcal{F}_{\rm av}(y_{\rm av}), \label{eq:dotXav_event_1_siso}
\end{align}
with
\begin{align}
\mathcal{F}_{\rm av}(y_{\rm av})
=
\lim_{T \to \infty} \frac1T \Bigg(
\int_{\bar{t}}^{\bar{t}+T}
\!\!\!\!\mathcal{F}(\tau,y_{\rm av})d\tau
+ \!\!\!\!
 \sum_{\bar{t} {\leq} \bar{t}_k(y) {<} \bar{t} + T}\!\!\!\! \mathcal{I}_k(y)
\Bigg).
\end{align}
This yields the average closed-loop dynamics together with the natural average counterpart of the original event-triggered implementation, given by
\begin{align}
\frac{d\hat{G}_{\text{av}}}{d\bar{t}}(\bar{t})&=\frac{1}{\omega}H^{\ast}K\hat{G}_{\rm{av}}(\bar{t})+\frac{1}{\omega}H^{\ast}Ke_{\rm{av}}(\bar{t})\,, \label{eq:dotHatGav_event_1_siso} \\
\frac{d\tilde{\theta}_{\text{av}}}{d\bar{t}}(\bar{t})&=\frac{1 }{\omega}H^{\ast}K\tilde{\theta}_{\rm{av}}(\bar{t})+\frac{1}{\omega}Ke_{\rm{av}}(\bar{t})\,, \label{eq:dotTildeThetaAv_event_1_siso}
\end{align}
where the average update error is given by
\begin{align}
e_{\text{av}}(\bar{t})&=\hat{G}_{\text{av}}(\bar{t}_{k})-\hat{G}_{\text{av}}(\bar{t})\,. \label{eq:Eav_event_1_siso} 
\end{align}
Note that the impulse average does not contribute to the average dynamics because the impulse map resets the error to zero, and over a period, the average error tends to zero.
The averaging method allows one to characterize the behavior of the nonautonomous system \eqref{eq:dotX_event_siso} through its autonomous average counterpart \eqref{eq:dotXav_event_1_siso}.
Treating the slow variables $\hat{G}(\bar{t})$, $\tilde{\theta}(\bar{t})$, and the transmission error $e(\bar{t})$ as frozen over one excitation period, the oscillatory terms in \eqref{eq:dotHatGav_event_3_siso} and \eqref{eq:dotTildeTheta_2_event_siso} have zero mean under Assumption~\ref{hipotese_1}.

Accordingly, the average triggering law is obtained by replacing
$(\hat{G},e)$ in the original triggering condition with
$(\hat{G}_{\rm av},e_{\rm av})$. Moreover, from (\ref{eq:hatG_20240302_3}),
\begin{align}
\hat{G}_{\text{av}}(\bar{t})= H^{\ast}\tilde{\theta}_{\text{av}}(\bar{t})\,. \label{eq:hatGav_event_1_siso}
\end{align}
Since $H^{*}$ is nonsingular, $\hat{G}_{\rm av}(\bar{t})$ and $\tilde{\theta}(\bar{t})$ are in one-to-one correspondence.
Hence, the following ETM can be established for the average system:
\begin{align}
    \hspace{-0.5cm}\bar{t}_{k+1}&{=}\inf\left\{\bar{t} \in\mathbb{R}^{+}: \bar{t}\!>\!\bar{t}_{k} \wedge  \Gamma (\hat{G}_{\rm{av}},e_{\rm{av}}) <0 \right\}\,, \label{eq:tk+1_event_av_siso}
\end{align}
with 
\begin{align}
\Gamma(\hat{G}_{\rm{av}},e_{\rm{av}}) = \hat{G}^{\top}_{\rm{av}}(\bar{t})Q_{\rm{G}}\hat{G}_{\rm{av}}(\bar{t}) -e^{\top}_{\rm{av}}(\bar{t})Q_{\rm{e}}e_{\rm{av}}(\bar{t}).  \label{eq:gamma_event_1_siso}
\end{align}
Moreover, the average event-triggered control law is
\begin{align}
    u^{\rm{av}}(\bar{t})\!=\!K\hat{G}_{\rm{av}}(\bar{t}_{k}) \,, \quad \forall \bar{t} \!\in\! \lbrack \bar{t}_{k}\,, \bar{t}_{k+1}\phantom{(}\!\!)\,, \quad k\!\in\!\mathbb{N}\,. \label{eq:U_MD3}
\end{align}


\section{Co-Design Condition and Stability Analysis} \label{sec:stblt_siso}

This section presents the main theoretical results of this paper. First, we present a co-design condition for the joint synthesis of the controller gain and triggering mechanism based on the average system. Then, we prove the stability analysis of the original system using the averaging method~\cite{KPS:2012} and the Zeno-freeness property. Finally, an optimization problem is provided to reduce the number of transmissions.

\subsection{Co-Design Condition}
The following lemma provides a condition to simultaneously design a gain matrix $K$ and the triggering matrices $Q_{\rm G}$ and $Q_{\rm e}$ such that the origin of the average closed-loop system is exponentially stable.
\begin{lemma} \label{lema1}
    Consider the average system \eqref{eq:dotHatGav_event_1_siso} and assume that Assumptions~\ref{hipotese_1} and \ref{hipotese_2} hold. For a given scalar $\eta > 0$, if there exist positive definite matrices $W, ~\tilde{Q}_G, ~\tilde{Q}_e \in \mathbb{R}^{n\times n}$ and a matrix $Z \in \mathbb{R}^{n \times n}$ such that the following LMI holds:
    \begin{align}
    {
    \begin{bmatrix}
    Z^\top H_i + H_i Z + 2\eta W
    & H_i Z
    & W \\
    Z^\top H_i
    & -\tilde Q_{\rm{e}}
    & 0 \\
    W
    & 0
    & -\tilde Q_{\rm{G}}
    \end{bmatrix}
    }
    \!\prec 0, \, i \!=\! 1,\ldots,N\,.
    \label{lema2:condition}
    \end{align}
Then, the average system is exponentially stable with decay rate $\eta > 0$, that is: 
\begin{align}
     \|\hat{G}_{\mathrm{av}}(t)\| \leq \kappa e^{-\eta t} \|\hat{G}_{\rm{av}}(0)\|, \label{eq:ghat_solution}
\end{align}
with $\kappa = \sqrt{\lambda_{\max}(P)/\lambda_{\min}(P)}$, $P = W^{-1}$, $K = ZW^{-1}$, $Q_{\rm G} = \tilde{Q}_{\rm G}^{-1}$, and $Q_{\rm e} = W^{-1}\tilde{Q}_{\rm e}W^{-1}$.
\end{lemma}
\begin{IEEEproof}
Assume that condition (\ref{lema2:condition}) holds. 
By performing a congruence transformation on \eqref{lema2:condition} with $\mbox{diag}(W^{-1}, W^{-1}, I)$ and defining $P = W^{-1}$, $K = ZW^{-1}$, $Q_{\rm G} = \tilde{Q}_{\rm G}^{-1}$, and $Q_{\rm e} = P^{-\top}\tilde{Q}_{\rm e}P^{-1}$, and then applying the Schur complement,
one obtains
\begin{align}
    \begin{bmatrix}
    K^\top H^{\ast}P + PH^{\ast}K + 2\eta P + Q_{\rm{G}}
    & PH^{\ast \top} K\\
    K^\top H^{\ast}P
    & -Q_{\rm{e}}
    \end{bmatrix} \prec 0 \,,
    \label{lema2:eq2}
\end{align}
provided that $H^{\ast} \in \mathcal{H}$. 
Premultiplying the matrix inequality in \eqref{lema2:eq2} by $[\hat{G}_{\rm av}^{\top} ~ e_{\rm av}^{\top}]$ and postmultiplying it by its transpose yields
\begin{align} 
\hat{G}_{\rm{av}}^\top (K^\top H^{\ast} P +  \nonumber
& P H^{*} K + 2\eta P + Q_{\rm{G}})\hat{G}_{\rm{av}} \nonumber \\ 
& + 2\hat{G}_{\rm{av}}^\top P H^{*} K e_{\rm{av}} - e_{\rm{av}}^\top Q_{\rm{e}} e_{\rm{av}} < 0\,.
\end{align}
Thus, it follows from the ETM~\eqref{eq:gamma_event_1_siso}--\eqref{eq:tk+1_event_av_siso} that
\begin{align}
    \dot{V}(\hat{G}_{\rm av}) \leq - 2\eta V(\hat{G}_{\rm av}) < 0\,, \label{eq:v_gav}
\end{align}
where $V(\hat{G}_{\rm av}) =  \hat{G}_{\rm{av}}^\top P \hat{G}_{\rm{av}}$
is a Lyapunov function that certifies the exponential stability of the origin of the average closed-loop system. Applying the Comparison Lemma \cite{K:2002} to (\ref{eq:v_gav}) yields $V(\hat{G}_{\rm av}(t)) \leq e^{-2 \eta t}V(\hat{G}_{\rm av}(0))$.
By means of the Rayleigh-Ritz inequality \cite{K:2002},
$\lambda_{\min}(P)\,\|\hat{G}_{\rm av}\|^{2} \leq\ V(\hat{G}_{\rm av})\leq\
\lambda_{\max}(P)\,\|\hat{G}_{\rm av}\|^{2}$,
it can be shown that~\eqref{eq:ghat_solution} holds with 
where $\kappa = \sqrt{\lambda_{\max}(P)/\lambda_{\min}(P)}$. Hence, the origin of the system is exponentially stable.
\hfill $\blacksquare$
\end{IEEEproof}

\begin{remark}
    \label{rem:diag}
    The co-design approach jointly synthesizes a controller gain $K$ with a full structure and the triggering matrices $Q_G$ and $Q_e$ while enforcing a prescribed exponential decay rate $\eta$. Note that a $K$ with diagonal structure can be enforced by solving the co-design condition of Lemma~\ref{lema1} with diagonal matrices $Z$ and $W$. 
\end{remark}

\subsection{Stability and Convergence Analysis}
\begin{theorem} \label{thm:NETESC_1_siso}
Consider the average closed-loop dynamics given by \eqref{eq:dotHatGav_event_1_siso}--\eqref{eq:dotTildeThetaAv_event_1_siso}, together with the average triggering mechanism \eqref{eq:tk+1_event_av_siso}, updated according to the control law \eqref{eq:U_MD3}. Under the conditions of Lemma~\ref{lema1}, with $\Gamma(\hat{G}_{\rm av},e_{\rm av})$ defined\linebreak in (\ref{eq:gamma_event_1_siso}), it follows that, for sufficiently large $\omega\!>\!0$ and sufficiently small initial condition $\theta(0)$, the equilibrium $(\hat{G}_{\rm av},\tilde{\theta}_{\rm av})\!=\!(0,0)$ is locally exponentially stable, and the norms of the input-output signals 
satisfy
\begin{align}
\|\theta(t) - \theta^{\ast}\| & \leq \kappa_{\theta} e^{-\eta t}\|\theta(0) - \theta^{\ast}\| + \mathcal{O}\left(a + \frac{1}{\omega}\right)\,, \label{eq:normTheta_thm1_siso} \\ 
|y(t) - Q^{\ast}| &\leq Me^{-\eta t} + 2\kappa_{\theta} e^{-\eta t}\|\theta(0) - \theta^{\ast}\|\mathcal{O}\left(a + \frac{1}{\omega}\right) \nonumber \\
&+ \mathcal{O}\left(a^{2} + \frac{1}{\omega^{2}} \right)\,, \label{eq:normY_thm1_siso}
\end{align}
where the constant $M$ depends on the initial condition $\theta(0)$.
\end{theorem}
\begin{IEEEproof}
From \eqref{eq:hatGav_event_1_siso}, since $H^{*}$ is invertible the following bound holds:
\begin{align}
   \|\tilde{\theta}_{\rm av}(\bar{t})\| \leq \|H^{\ast -1}\|\|\hat{G}_{\rm av}(\bar{t})\|\,. \label{eq:theta_bound}
\end{align}
Since $\hat{G}_{\rm av}(0)=H^{\ast}\tilde{\theta}_{\rm av}(0)$, it follows that
$\|\hat{G}_{\rm av}(0)\|\leq\|H^{\ast}\|\,\|\tilde{\theta}_{\rm av}(0)\|$, which together (\ref{eq:ghat_solution}) and \eqref{eq:theta_bound} yields
\begin{align}
    \| \tilde{\theta}_{\rm av}(\bar{t})\| \leq \kappa_{\theta} e^{-\eta t}\|\tilde{\theta}_{\rm av}(0)\|\,, \label{eq:tilde_theta_av_exp}
\end{align}
where $\kappa_{\theta} = \kappa\|H^{\ast -1}\|\|H^{\ast }\|$. Since the differential equation (\ref{eq:dotTildeTheta_2_event_siso}) has a discontinuous right-hand side and is $T$-periodic in $t$, and since the average dynamics $\tilde{\theta}_{\rm av}(\bar{t})$ in (\ref{eq:tilde_theta_av_exp}) are asymptotically stable, invoking the Averaging Theorem \cite{P:1979,KPS:2012} guarantees that
\begin{align}
    \|\tilde{\theta}(t) - \tilde{\theta}_{\rm av}(t)\| \leq \mathcal{O}\left(\frac{1}{\omega}\right)\,.
\end{align}
Applying the triangle inequality yields
\begin{align}
    \|\tilde{\theta}(t)\| \leq \kappa_{\theta} e^{-\eta t}\|\tilde{\theta}_{\rm av}(0)\|+ \mathcal{O}\left(\frac{1}{\omega}\right)\,.
\end{align}
Similarly, the same arguments can be applied to $\hat{G}(t)$, and the Averaging Theorem \cite{P:1979,KPS:2012} ensures that
\begin{align}
    \|\hat{G}(t) - \hat{G}_{\rm av}(t)\| \leq \mathcal{O}\left(\frac{1}{\omega}\right)\,,
\end{align}
by the triangle inequality, it follows that
\begin{align}
    \|\hat{G}(t)\| \leq \kappa e^{-\eta t}\|\hat{G}_{\rm av}(0)\|+ \mathcal{O}\left(\frac{1}{\omega}\right)\,.
\end{align}
Now, from (\ref{eq:theta_event_siso}) and (\ref{eq:thetaTilde}), it follows that
\begin{align}
    \theta(t) - \theta^{\ast} = \tilde{\theta}(t) + S(t)\,.
\end{align}
By applying the Euclidean norm,
\begin{align}
     \|\theta(t) - \theta^{\ast}\| &= \|\tilde{\theta}(t) + S(t)\| \leq \|\tilde{\theta}(t)\| + \| S(t)\|\,,\nonumber \\
     \|\theta(t) - \theta^{\ast}\| & \leq \kappa_{\theta} e^{-\eta t} \|\theta(0) - \theta^{\ast}\| + \mathcal{O}\left(a + \frac{1}{\omega}\right)\,.
     \label{eq:norm_tilde_exp}
\end{align}
For the output, the estimation error is defined as
\begin{align}
    \tilde{y}(t):= y(t) - Q^{\ast}\,,
\end{align}
by taking its norm and using the Cauchy--Schwarz inequality, one obtains
\begin{align}
    |y(t) - Q^{\ast}| {=} 
    \leq \frac{1}{2}\|H^{\ast}\| \|\theta(t) - \theta^{\ast}\| ^{2}.\label{eq:y_tilde}
\end{align}
Using (\ref{eq:norm_tilde_exp}), (\ref{eq:y_tilde}) can be rewritten as
\begin{align}
    |y(t) - Q^{\ast}| &\leq\frac{1}{2} \|H^{\ast}\| \Big[ \kappa_{\theta}^{2} e^{-2\eta t}\|\theta(0) {-} \theta^{\ast}\|^{2}\nonumber \\ 
     & + 2\kappa_{\theta} e^{-\eta t} \|\theta(0) - \theta^{\ast}\|\mathcal{O}\left(a + \frac{1}{\omega}\right) \nonumber \\
    &+ \mathcal{O}\left(a^{2} + \frac{2a}{\omega} + \frac{1}{\omega^{2}}\right) \Bigg]\,,
\end{align}
Since $e^{-2\eta t} \leq e^{-\eta t}$ for $\eta > 0$, and using Young's inequality we have $\frac{2a}{\omega} \leq a^{2} + \frac{1}{\omega^{2}}$, it follows that
\begin{align}
   |y(t) - Q^{\ast}| &{\leq} Me^{-\eta t} {+}  \mathcal{O}\left(a^{2} {+} \frac{1}{\omega^{2}} \right) \nonumber \\
   &{+} 2\kappa_{\theta} e^{-\eta t} \|\theta(0) {-} \theta^{\ast}\|\mathcal{O}\left(a {+} \frac{1}{\omega}\right),
\end{align}
where
$M = \frac{1}{2}\|H^{\ast}\|\kappa_{\theta}^{2}\|\theta(0) - \theta^{\ast}\|^{2}$.
Therefore, \eqref{eq:normTheta_thm1_siso} and \eqref{eq:normY_thm1_siso} are obtained, completing the proof. 
\hfill $\blacksquare$
\end{IEEEproof}

\begin{proposition} \label{propzeno}
For $\omega > 0$ sufficiently large, the inter-event interval satisfies $t_{k+1} - t_{k} > \tau^*$, for all $k \in \mathbb{N}$, thus excluding Zeno behavior, with
\begin{align}
\tau^\ast
=
\frac{
1-\mathcal O\!\left(\frac{1}{\omega}\right)
}{
\|H^{\ast}K\|
\beta^2
\left(
1+\beta^{-1}-\mathcal O\!\left(\frac{1}{\omega}\right)
\right)
}\,, \quad \beta>0.\label{eq:miet}
\end{align}
\end{proposition}

\begin{IEEEproof}
From the average closed-loop system, particularly the dynamics of
$\hat{G}_{\rm av}(\bar{t})$ in
\eqref{eq:dotHatGav_event_1_siso} and the triggering condition
$\Gamma(\hat{G}_{\rm av},e_{\rm av}) < 0$ in
\eqref{eq:tk+1_event_av_siso}, it follows that
\begin{align}
e_{\rm av}^{\top}Q_{\rm e}e_{\rm av}
>
\hat G_{\rm av}^{\top}Q_{\rm G}\hat G_{\rm av}.
\end{align}
Following \cite{G:2014}, define
\begin{align}
\phi_{\rm av}(\bar t)
=
\beta
\frac{\|e_{\rm av}(\bar t)\|}
{\|\hat G_{\rm av}(\bar t)\|},
\qquad
\beta
=
\sqrt{\frac{\lambda_{\max}(Q_{\rm e})}
{\lambda_{\min}(Q_{\rm G})}}. \label{eq:beta}
\end{align}
A lower bound on the inter-execution time is given by the time required for
$\phi_{\rm av}$ to evolve from $0$ to $1$. Using
\eqref{eq:dotHatGav_event_1_siso} and \eqref{eq:Eav_event_1_siso}, 
straightforward 
calculations yield
\begin{align}
\frac{d\phi_{\rm av}}{d\bar t}
\leq
\frac{
\|H^{\ast}K\|}{\omega}
\beta
\left(
\frac1\beta+\phi_{\rm av}
\right)^2 .
\end{align}
According to the Comparison Lemma \cite{K:2002}, an upper bound $\tilde{\phi}_{\rm av}(\bar{t})$ for $\phi_{\rm av}(\bar{t})$ is given by
\begin{align}
\phi_{\rm av}(\bar{t}) \le \tilde{\phi}_{\rm av}(\bar{t}), 
\qquad 
\phi_{\rm av}(0)=\tilde{\phi}_{\rm av}(0)=0,
\end{align}
an upper bound $\tilde{\phi}_{\rm av}(\bar{t})$ for $\phi_{\rm av}(\bar{t})$ is obtained as the solution to
\begin{align}
\tilde\phi_{\rm av}(\bar t)
=
\frac{\beta^{-1}}
{1-\frac{\|H^{\ast}K\|}{\omega}\beta^2\bar t}
-\frac1\beta .
\end{align}
Since $\phi_{\rm av}$ is the average counterpart of
$\phi(t)=\beta\|e(t)\|/\|\hat G(t)\|$, the Averaging Theorem
\cite{P:1979,KPS:2012} implies
\begin{align}
|\phi(t)-\phi_{\rm av}(t)|
\leq
\mathcal O\!\left(\frac1\omega\right).
\end{align}
Hence, by the triangle inequality,
\begin{align}
\phi(t)
\le
\tilde\phi_{\rm av}(t)
+
\mathcal O\!\left(\frac{1}{\omega}\right).
\end{align}
Therefore, the original system admits the positive inter-execution-time
lower bound given in~\eqref{eq:miet}, which excludes Zeno behavior since $1{+}\beta^{-1}{-}\mathcal O\!\left(\frac{1}{\omega}\right) > 0$ for sufficiently large $\omega$.
\hfill $\blacksquare$
\end{IEEEproof}

\subsection{Optimization Problem} \label{optproblem}

From the minimum inter-event time $\tau^*$ established in~\eqref{eq:miet}, it can be noticed that, for $\omega$ sufficiently large, the value of $\tau^*$ can be increased by minimizing $\beta = \sqrt{{\lambda_{\max}(Q_{\rm e})}/
{\lambda_{\min}(Q_{\rm G})}}$ as established in~\eqref{eq:beta}. This can be accomplished by minimizing the eigenvalues of $Q_{\rm e}$ and maximizing the eigenvalues of $Q_{\rm G}$. Thus, similar to~\cite{CP:2022}, the following convex optimization problem is provided to solve the co-design condition established in Lemma~\ref{lema1} by maximizing the minimum inter-event time and, consequently, reducing the number of transmissions:
\begin{align}
\min_{W,Z,Q, \tilde{Q}_{\rm e}, \tilde{Q}_{\rm G}}
\quad & \mathrm{tr}(\tilde Q_{{e}} + \tilde Q_{{G}} + Q)
\label{eq:min_tr} \\
\text{subject to} \quad &
\begin{bmatrix}
- Q & I \\
I & -W
\end{bmatrix}
\prec 0\,, \label{eq:min_matriz} \\
&\text{and the LMI \eqref{lema2:condition}}. \nonumber
\end{align}
where $Q \succ 0$, $W \succ 0,\tilde{Q}_{\rm G} \succ 0$, $\tilde{Q}_{\rm e}\succ 0$, and $Z$, are all decision variables. The optimization problem in (\ref{eq:min_tr}) minimizes the traces of $Q$, $\tilde{Q}_{\rm e}$, and $\tilde{Q}_{\rm G}$, which tends to reduce their eigenvalues. 
By applying the Schur complement to (\ref{eq:min_matriz}), one obtains $W^{-1} \prec Q$. Hence, minimizing $\mathrm{tr}(Q)$ tends to minimize the eigenvalues of $W^{-1}=P \succ 0$. Moreover, by minimizing the eigenvalues of $\tilde{Q}_{\rm G}$, it maximizes the eigenvalues of $Q_{\rm G} = \tilde{Q}_{\rm G}^{-1}$. Also, minimizing the eigenvalues of $Q$, $\tilde{Q}_{\rm e}$, minimizes the eigenvalues of $Q_e = W^{-1} \tilde{Q}_e W^{-1}$, thereby minimizing $\beta = \sqrt{{\lambda_{\max}(Q_{\rm e})}/
{\lambda_{\min}(Q_{\rm G})}}$, increasing the inter-event intervals, and, consequently, reducing the number of transmissions. 

\section{Simulation Results} \label{sec:sim_siso}

We consider the nonlinear map defined in (\ref{eq:y_v2}) with a two-dimensional input $\theta(t) \in \mathbb{R}^{2}$ and a scalar output $y(t) \in \mathbb{R}$, with an unknown Hessian matrix whose values belong to the following polytopic set defined by the vertices: $H_{1} \!=\! (1 \!-\! \bar{\sigma})H_{0}$ and $H_{2} \!=\! (1 \!+\! \bar{\sigma})H_{0}$, where $\bar{\sigma} > 0$ is the uncertainty with respect to a nominal Hessian matrix $H_0$ given by \cite{GKN:2012}:
\begin{align}
H_{0} =
\begin{bmatrix}
100 & 30 \\
30 & 20
\end{bmatrix} \succ 0.
\end{align}
The unknown mapping is characterized by $Q^\ast = 100$ and $\theta^\ast = [2 ~4]^{\top}$. The simulations were performed with $\bar{\sigma}=0.6$, using the vertices $H_{1}$ and $H_{2}$, perturbation frequencies $\omega_{1}=1$ rad/s and $\omega_{2}=7$ rad/s, amplitudes $a_{1}=a_{2}=0.1$, and the initial condition $\hat{\theta}(0)=[2.5~5]^{\top}$. These parameters were selected based on \cite{ROKT:2026}.

Initially, we evaluate how the knowledge of curvature influences the number of transmissions. For that purpose, we fixed the decay rate $\eta = 1$ and solved the optimization problem~\eqref{eq:min_tr} for 10 values of $\bar{\sigma}$ linearly spaced in the interval from $0$ to $0.99$. Note that smaller values of $\bar{\sigma}$ indicate more knowledge about the Hessian matrix (curvature). To evaluate the advantages of using a full gain matrix, we repeated this experiment using a full gain matrix and a diagonal gain matrix. The co-design with a diagonal structure is solved according to Remark~\ref{rem:diag}. The results are depicted in Fig.~\ref{fig:3}, where the optimal values of the objective function of~\eqref{eq:min_tr} obtained with the full and diagonal gain matrices are shown. In both cases, for low uncertainty levels, the cost is small and increases monotonically as $\bar{\sigma}$ grows. It clearly indicates that more curvature knowledge results in fewer transmissions. Recall that smaller values of $\mathrm{tr}(\tilde{Q}_e+\tilde{Q}_G+Q)$ indicate greater reductions of the number of transmissions. 
\begin{figure}[!ht]
	\centering
    \subfigure{\includegraphics[width=0.8\columnwidth]{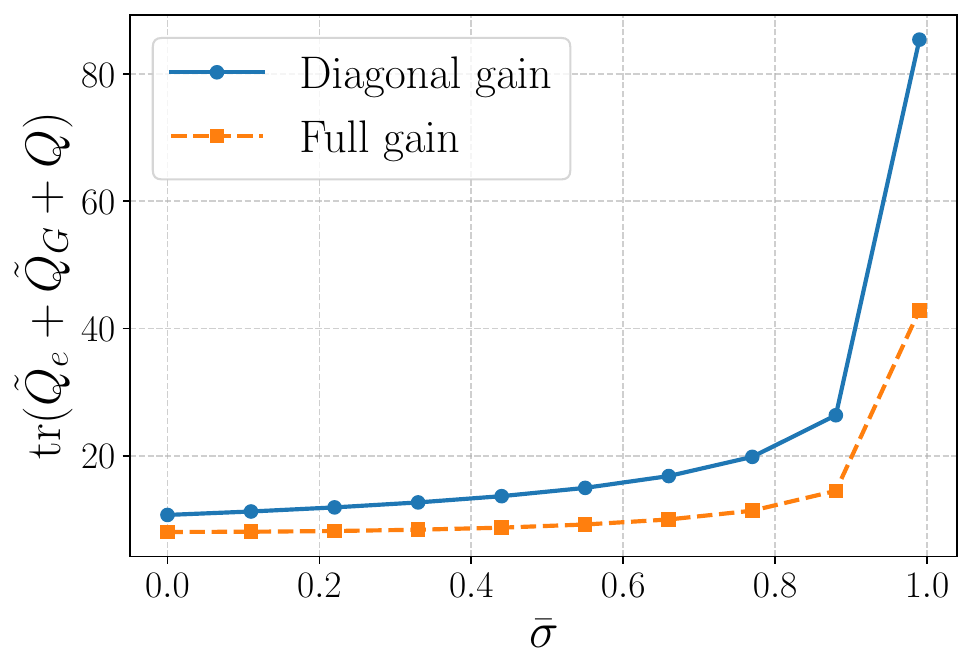}}
    \caption{Cost function of the optimization problem~\eqref{eq:min_tr} solved for different values of $\bar{\sigma}$, considering diagonal and full gain structures.}
    \label{fig:3}
\end{figure}

Moreover, for all values of $\bar{\sigma}$, the co-design condition solved with a full gain matrix leads to smaller values of $\mathrm{tr}(\tilde{Q}_e+\tilde{Q}_G+Q)$ than using a diagonal structure.  Therefore, within the proposed co-design framework, no diagonal gain can simultaneously satisfy the decay rate and the same triggering threshold achieved by the corresponding full-matrix gain. Since both the diagonal and full gain designs are synthesized within the same co-design framework, the performance gap observed in Fig.~\ref{fig:3} is related to the use of a full gain structure. The superior performance of the full gain matrix is primarily related to the extra degrees of freedom to solve the optimization problem provided by the extra decision variables in the full gain structure. 

By fixing the decay rate $\eta = 1$, the Hessian uncertainty bound $\bar\sigma = 0.6$, and setting a triggering-threshold bound $\bar{J} = 10$, such that $\mathrm{tr}(\tilde{Q}_e+\tilde{Q}_G+Q) \leq \bar{J}$, and solving~\eqref{eq:min_tr} {with the full gain matrix, we obtain
\begin{align*}
    K = 
    \begin{bmatrix}
        -0.0795  &  0.1193 \\
        0.1193  & -0.3977
    \end{bmatrix}, 
\end{align*}
$Q_G = 1.8974 I_2$, and $Q_e = 10.3763 I_2$, with the $\mathrm{tr}(\tilde{Q}_e+\tilde{Q}_G+Q) = 9.4889$.}
In this setting, only the co-design with a full gain is \textit{feasible}, meaning that a diagonal gain cannot realize the same decay rate with the same triggering threshold. The simulation results for the closed-loop system with the co-designed event-triggered ES control scheme with full control gain structure are shown in Fig.~\ref{fig:simulation}. {In particular, Fig. \ref{fig:simulation}(a) and Fig. \ref{fig:simulation}(b) show the evolution of the optimization variables, where both states converge to their optimal values for the Hessian vertices $H_{1}$ and $H_{2}$, respectively. Fig. \ref{fig:simulation}(c) and Fig. \ref{fig:simulation}(d) illustrate the corresponding output trajectories, confirming convergence to the optimal objective value $Q^{*}$. The control inputs in Fig. \ref{fig:simulation}(e) and Fig. \ref{fig:simulation}(f) approach zero as the optimum is reached. Finally, Fig. \ref{fig:simulation}(g) and Fig \ref{fig:simulation}(h) present the event-triggering instants. The simulation with the vertex $H_{2}$, which represents the Hessian vertex associated with the upper uncertainty bound, requires more communication updates than $H_{1}$ (lower bound), 104 updates versus 38 updates, while still preserving the prescribed convergence rate.} This illustrates that the synthesized gain behaves differently between vertices. However, the convergence is theoretically ensured for any $H^\ast \in \mathcal{H}$.
\begin{figure}[!ht]
    \vspace{-0.35cm}
    \hspace{-0.95cm}\includegraphics[width = 1.2\columnwidth]{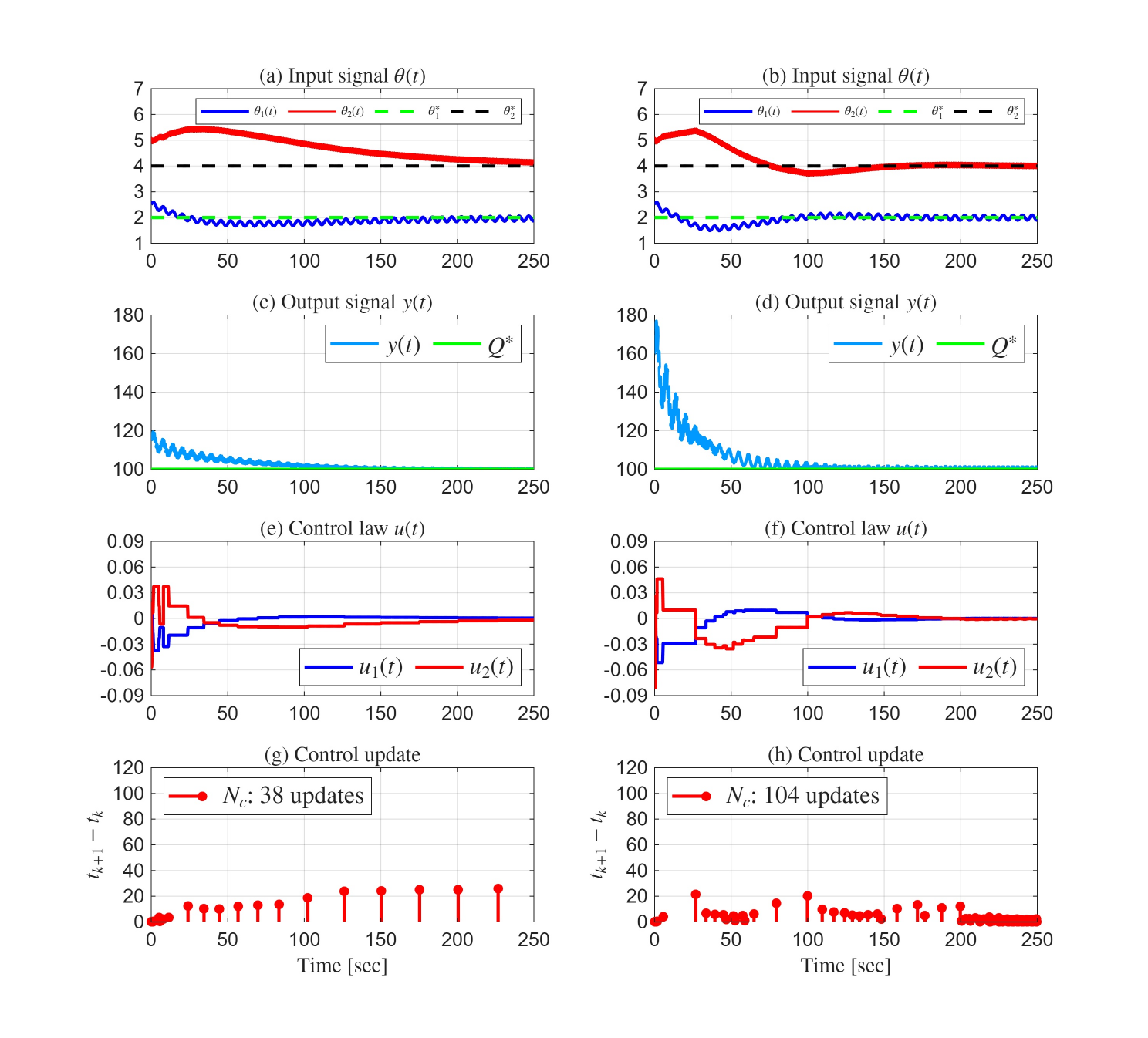}
    \vspace{-1.0cm}
    \caption{Simulation results of the event-triggered ES using a full-matrix gain for the vertex $H_{1}$ (left column) and for the vertex $H_{2}$ (right column).}
    \label{fig:simulation}
\end{figure}

\section{Conclusion} \label{sec:concl_siso}

This paper presented a co-design framework for an event-triggered ES controller, jointly synthesizing a stabilizing feedback gain and a triggering law.  The proposed co-design strategy exploits prior Hessian information to guarantee exponential stability of the system with a given decay rate while reducing the communication load. Using averaging theory and Lyapunov analysis, local stability around the optimum point was proved. Moreover, we showed that the closed-loop solutions do not present Zeno behavior. By comparing diagonal and full gain designs within the co-design framework, we isolated the contribution of off-diagonal elements. The numerical results illustrated that the off-diagonal structure allows for achieving superior triggering performance, as diagonal gains fail to achieve the same decay rate and triggering threshold even with the same co-design methodology.

The most significant practical implication of Fig.~\ref{fig:3} is that increasing the uncertainty polytope leads to a higher trigger count. These results quantitatively demonstrate that more accurate prior curvature information can be directly translated into reduced communication through joint controller-trigger synthesis.









\bibliographystyle{IEEEtranS} 
\bibliography{references}

@article{APDNH:2018,
  author = {M. Abdelrahim and R. Postoyan and J. Daafouz and D. Nešić and M. Heemels},
  title = {Co-design of output feedback laws and event-triggering conditions for the {$\mathcal{L}_2$}-stabilization of linear systems},
  journal = {Automatica},
  volume = {87},
  pages = {337--344},
  year = {2018}
}

@book{AK:2003,
  author = {K. B. Ariyur and M. Krstić},
  title = {Real-Time Optimization by Extremum-Seeking Control},
  address = {Canada},
  publisher = {Wiley},
  year = {2003}
}

@inproceedings{AB:1999,
  author = {K. J. {\r A}str{\"o}m and B. P. Bernhardsson},
  title = {Comparison of periodic and event based sampling for first-order stochastic systems},
  booktitle = {IFAC World Congress},
  volume = {32},
  pages = {5006--5011},
  year = {1999}
}

@article{CP:2022,
  author = {P. H. S. Coutinho and R. M. Palhares},
  title = {Codesign of dynamic event-triggered gain-scheduling control for a class of nonlinear systems},
  journal = {IEEE Trans. Autom. Control},
  volume = {67},
  number = {8},
  pages = {4186--4193},
  year = {2022}
}

@article{CPBPP:2026,
  author = {P. H. S. Coutinho and P. S. P. Pessim and I. Bessa and M. L. C. Peixoto and R. M. Palhares},
  title = {Handling asynchronous scheduling functions in periodic event-triggered gain-scheduled control with guaranteed polytopic inclusion},
  journal = {IEEE Transactions on Cybernetics},
  volume = {56},
  number = {8},
  pages = {4206--4218},
  year = {2026}
}

@article{GKN:2012,
  author = {A. Ghaffari and M. Krstić and D. Nešić},
  title = {Multivariable {Newton}-based extremum seeking},
  journal = {Automatica},
  volume = {48},
  pages = {1759--1767},
  year = {2012}
}

@article{G:2014,
  author = {A. Girard},
  title = {Dynamic triggering mechanism for event-triggered control},
  journal = {IEEE Trans. Autom. Control},
  volume = {60},
  pages = {1992--1997},
  year = {2014}
}

@article{GUAY:2021,
  author = {M. Guay},
  title = {Uncertainty estimation in extremum seeking control of unknown static maps},
  journal = {IEEE Control Systems Letters},
  volume = {5},
  number = {4},
  pages = {1115--1120},
  year = {2021}
}

@book{K:2002,
  author = {H. K. Khalil},
  title = {Nonlinear Systems},
  address = {Upper Saddle River, NJ, USA},
  publisher = {Prentice Hall},
  year = {2002}
}

@article{KW:2000,
  author = {M. Krstić and H. H. Wang},
  title = {Stability of extremum seeking feedback for general nonlinear dynamic systems},
  journal = {Automatica},
  volume = {36},
  pages = {595--601},
  year = {2000}
}

@book{TRoux:2022,
  author = {T. R. Oliveira and M. Krstić},
  title = {Extremum Seeking through Delays and {PDE}s},
  address = {Philadelphia, PA, USA},
  publisher = {Society for Industrial and Applied Mathematics (SIAM)},
  year = {2022}
}

@article{CSm:2026,
  author = {T. R. Oliveira and M. Krstić and T. Basar},
  title = {Extremum and Nash equilibrium seeking with delays and {PDE}s: Designs \& applications into the second century of extremum-seeking control},
  journal = {IEEE Control Systems},
  volume = {46},
  number = {2},
  pages = {39--87},
  year = {2026}
}

@article{OTK:2017,
  author = {T. R. Oliveira and M. Krstić and D. Tsubakino},
  title = {Extremum seeking for static maps with delays},
  journal = {IEEE Trans. Autom. Control},
  volume = {62},
  number = {4},
  pages = {1911--1926},
  year = {2017}
}

@article{POPF:2020,
  author = {P. Paz and T. R. Oliveira and A. V. Pino and A. P. Fontana},
  title = {Model-free neuromuscular electrical stimulation by stochastic extremum seeking},
  journal = {IEEE Transactions on Control Systems Technology},
  volume = {28},
  number = {1},
  pages = {238--253},
  year = {2020}
}

@article{P:1979,
  author = {V. A. Plotnikov},
  title = {Averaging of differential inclusions},
  journal = {Ukrainian Mathematical Journal},
  volume = {31},
  pages = {454--457},
  year = {1979}
}

@article{VHPR:2025,
  author = {V. H. P. Rodrigues and T. R. Oliveira and L. Hsu and M. Diagne and M. Krstić},
  title = {Event-triggered and periodic event-triggered extremum seeking control},
  journal = {Automatica},
  volume = {174},
  number = {Art. no. 112161},
  year = {2025}
}

@article{ROKT:2026,
  author = {V. H. P. Rodrigues and T. R. Oliveira and M. Krstic and P. Tabuada},
  title = {Event-triggered {Newton} extremum seeking for multivariable optimization},
  year = {2026},
  note = {Available: \url{https://arxiv.org/abs/2601.14416}}
}

@article{SHPMG:2023,
  author = {P. K. Shahri and B. HomChaudhuri and S. S. Pulugurtha and A. Mesbah and A. H. Ghasemi},
  title = {Traffic congestion control using distributed extremum seeking and filtered feedback linearization control approaches},
  journal = {IEEE Control Systems Letters},
  volume = {7},
  pages = {1003--1008},
  year = {2023}
}

@article{AS:2024,
  author = {A. Scheinker},
  title = {100 years of extremum seeking: A survey},
  journal = {Automatica},
  volume = {161},
  number = {Art. no. 111481},
  year = {2024}
}

@article{T:2007,
  author = {P. Tabuada},
  title = {Event-triggered real-time scheduling of stabilizing control tasks},
  journal = {IEEE Trans. Autom. Control},
  volume = {52},
  pages = {1680--1685},
  year = {2007}
}

@article{ZHGDDYP:2020,
  author = {X. M. Zhang and Q. L. Han and X. Ge and D. Ding and L. Ding and D. Yue and C. Peng},
  title = {Networked control systems: A survey of trends and techniques},
  journal = {IEEE/CAA Journal of Automatica Sinica},
  volume = {7},
  number = {1},
  pages = {1--17},
  year = {2020}
}

@article{ZFO:2023,
  author = {Y. Zhu and E. Fridman and T. R. Oliveira},
  title = {Sampled-data extremum seeking with constant delay: A time-delay approach},
  journal = {IEEE Trans. Autom. Control},
  volume = {68},
  number = {1},
  pages = {432--439},
  year = {2023}
}

@article{KPS:2012,
title = {Overview of {V.A. Plotnikov’s} research on averaging of differential inclusions},
journal = {Physica D: Nonlinear Phenomena},
volume = {241},
number = {22},
pages = {1932-1947},
year = {2012},
note = {Dynamics and Bifurcations of Nonsmooth Systems},
issn = {0167-2789},
author = {S. Klymchuk and A. Plotnikov and N. Skripnik},
}

\end{document}